\documentclass[a4paper,11pt]{article}
\usepackage{amsmath,geometry,color}
\usepackage{graphicx} 
\usepackage{amsfonts,amsthm,amssymb}
\usepackage{amsfonts}
\usepackage{graphics}
\usepackage{pstricks}
\usepackage{graphicx}
\usepackage{epstopdf} 
\usepackage{blkarray} 
\usepackage{extarrows,chngpage,array,float} 
\usepackage{subcaption} 
\usepackage{amsmath} 
\usepackage{booktabs,setspace} 
\usepackage{hyperref}
\usepackage{cleveref}

\usepackage{tikz}
\usetikzlibrary{positioning}
\usepackage{amssymb}
\usepackage{latexsym,bm}
\usepackage{cite}

\newtheorem{theorem}{Theorem}[section]
\newtheorem{lemma}[theorem]{Lemma}

\newtheorem{corollary}[theorem]{Corollary}
\newtheorem{problem}[theorem]{Problem}

\newtheorem{example}[theorem]{Example}
\newtheorem*{theorem*}{Theorem}

\allowdisplaybreaks 
\usepackage{multicol}
\usepackage{makecell}
\usepackage{authblk}

\begin{document}
\title{\textbf{On the enumeration of connected sets in finite cylindrical lattice graphs}}
\author[a,b]{Hongxia Ma}
\author[a]{Xian'an Jin}
\author[a]{Meiqiao Zhang\thanks{Corresponding Author.\\\indent \hspace{0.18cm} E-mails:  hongxiama516@163.com, xajin@xmu.edu.cn, meiqiaozhang95@163.com.}}
	\affil[a]{\small School of Mathematical Sciences, Xiamen University, P. R. China}
	\affil[b]{\small School of Mathematical Sciences, Xinjiang Normal University, P. R. China}
\date{}
\maketitle

\begin{abstract}
A connected set in a graph is a non-empty set of vertices that induces a connected subgraph. In an infinite lattice, a connected set is often referred to as a lattice animal, whose enumeration up to isomorphism is a classical problem in both combinatorics and statistical physics. In this paper, we focus on the enumeration of connected sets in finite lattice graphs, providing a link between combinatorial counting and structural connectivity in the system.

For any positive integers $m,n$, let $N(P_m\times P_n)$ and $N(C_m\times P_n)$ denote the number of all connected sets in the $(m\times n)$-lattice graph $P_m\times P_n$ and $(m\times n)$-cylindrical lattice graph $C_m\times P_n $, respectively. In 2020, Vince derived enumeration formulas for $N(P_m\times P_2)$ and $N(C_m\times P_2)$, and highlighted the increasing difficulty of extending these calculation results to larger (cylindrical) lattice graphs. Recently, the authors of this paper have developed a method based on multi-step recurrence formulas to obtain the enumeration formula for $N(P_m\times P_n)$ with $m\le 4$. In this article, we apply a similar approach to derive the enumeration formula for $N(C_m\times P_n)$ with $m\le 7$. Further, for the general case, we establish an explicit and tight lower bound on the number of connected sets in the Cartesian product graph $G\times P_n $ for any connected graph $G$, by employing the transfer matrix method on a subclass of connected sets. Based on this, we perform an asymptotic analysis on several lattice graphs and show that $O(N(P_3\times P_n))=1.6694^{3n}$, $O(N(C_4\times P_n))=1.8014^{4n}$, and $O(N(C_5\times P_n))=1.7877^{5n}$.
\vspace{0.2cm} 
\end{abstract}

\section{Introduction}
\noindent
The enumeration of local structures in lattices has long been a topic of interest in both combinatorics and statistical physics. Classical examples include the dimer problem, which concerns the counting of perfect matchings~\cite{Lai,Yan}, as well as the enumeration of independent sets~\cite{SO,Calkin}, spanning forests and acyclic orientations~\cite{SC20201,Calkin2003}.
In particular, the enumeration of \textit{lattice animals} (finite sets of connected lattice sites or bonds) arises naturally in  studies of percolation, branching processes, and polymer networks, as these animals model connected structures such as clusters or polymer configurations. Most research in this area concentrates on counting distinct animals of a given size in an infinite lattice up to translation and other symmetries, and on analyzing the asymptotic behavior of these counts as the size of the animals tends to infinity~\cite{Luther,Jensen2000,Jensen2001}.
While these studies provide valuable insights into large-scale combinatorial behavior, many mysteries about lattices remain.

In this paper, we consider a related but complementary problem. Rather than focusing on an infinite lattice and distinct animals, we study the number of all animals within a finite lattice region, which reflects the connectivity of the finite lattice regarding its size and geometry. This research is also motivated by and originates from relevant studies on the enumeration of spanning trees in finite lattices~\cite{Daoud,Xiao}. From a statistical mechanics perspective, enumerating all animals in a finite region is similar to listing all microstates of a discrete system, with each animal corresponding to a distinct microconfiguration. In this way, our study connects combinatorial enumeration with finite-system statistical mechanics, linking structural counts directly to system connectivity and size effects. Furthermore, this approach may offer insights relevant to percolation and other lattice-based models in statistical physics.

For the remainder of this paper, we consider finite graphs only.
For any graph $G=(V(G),E(G))$, a \textit{connected set} $C$ of $G$ is a non-empty subset of $V(G)$ such that the subgraph of $G$ induced by $C$, denoted by $G[C]$, is connected. Clearly, when $G$ is a lattice, any connected set of $G$ is an animal and vice versa.
Let $\mathcal C(G)$ be the set of all connected sets of $G$, let $N(G)=|\mathcal C(G)|$, and let $c(G)=\sqrt[|V(G)|]{N(G)}$. 
Then $N(G)$ is the number of all connected sets in $G$, and $c(G)$ is a parameter characterizing the asymptotic behavior of $N(G)$, where $1<c(G)<2$ as shown in~\cite{Cambie,Haslegrave, KANG2018}.
 \par

In the literature, research on counting connected sets has been extensive and approached from various perspectives. For instance,
Kangas et al. \cite{KANG2018} determined upper bounds of $N(G)$ for graphs $G$ with the maximum degree at most 5; Luo et al. \cite{Luo} established sharp upper and lower bounds for $N(G)$ in terms of graph parameters  such as the chromatic number, stability number, matching number and connectivity; Cambie et al. \cite{Cambie} recently studied the maximum number of connected sets among all $d$-regular graphs for small $d$; Vince \cite{Vince22}
 conducted studies on several parameters derived from $N(G)$ and proved that \( P_n \) minimizes the average order of connected sets over all connected graphs of order $n$; Haslegrave \cite{Haslegrave} proved that the average connected set density is bounded away from 1 for graphs without vertices of degree 2; and Vince
 \cite{Vince23} showed that if \( G \) is a connected vertex transitive graph, then the average order of the connected sets of \( G \) is at least \( |V(G)|/2 \).

 Back to the problem we are concerned with, let $P_{n}$ and $C_{n}$ be the path and cycle of order $n$ for any positive integer $n$, respectively.
The \textit{Cartesian product} of any pair of graphs $G$ and $H$ is denoted by the graph $G\times H$, where $V(G\times H)=V(G)\times V(H)$ and $(x_{1},y_{1})(x_{2},y_{2})\in E(G\times H)$ if and only if either $x_{1}x_{2} \in E(G)$ and $y_{1}=y_{2}$, or $y_{1}y_{2} \in E(H)$ and $x_{1}=x_{2}$.
Then $P_{m}\times P_{n}$ and $C_{m}\times P_{n}$ are lattice graphs, and are known as the \textit{$(m\times n)$-lattice graph} and the \textit{$(m\times n)$-cylindrical lattice graph}, respectively.
See two examples in Figure~\ref{f1}.
\begin{figure}[h]
\centering
\scalebox{0.35}[0.35]{\includegraphics{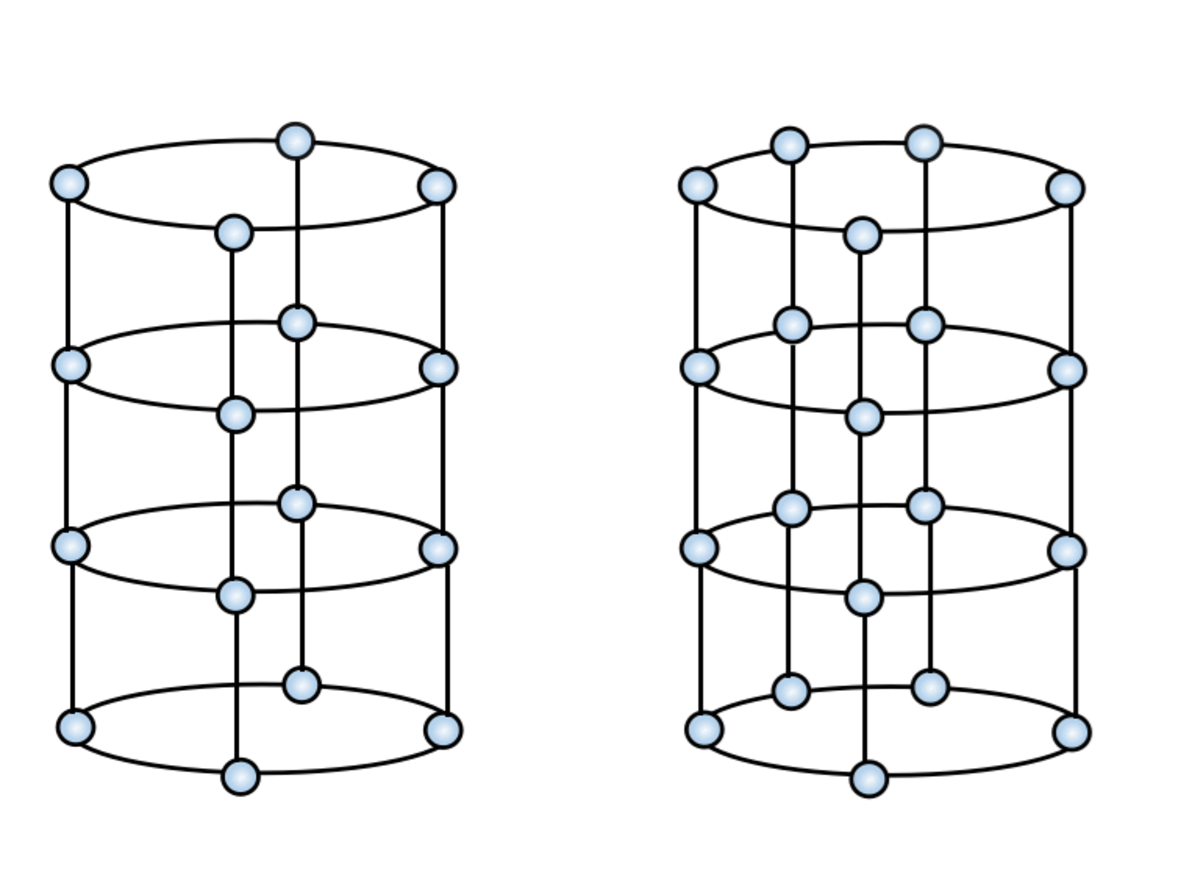}}
\caption{$C_4\times P_4$ and $C_5\times P_4$}
\label{f1}
\end{figure}
In 2020, Vince \cite{Vince2020} derived the following explicit enumeration formulas for connected sets in $P_n\times P_2$ and $C_n\times P_2 :$
$$N(P_{n}\times P_{2})=\frac{\beta(n+3)-4n-7}{2}~~\text{and}~~N(C_{n}\times P_{2})= 1-3n+2\beta(n)+3n\overline{\beta}(n),$$ where $\overline{\beta}(n)$ and $\beta(n)$ are Pell number and Pell-Lucas numbers.
Beyond these results, he posed the problem of finding a formula for general $N(P_n\times P_n)$ with noting the difficulty of solving this problem in \cite{Vince2020}, since determining the number of connected sets in an arbitrary graph is \#P-complete \cite{SUT}.

\begin{problem}[\cite{Vince2020}]\label{prob}
Find a formula for the number of connected sets in the
$(n\times n)$-lattice graph.
\end{problem}

As a preliminary study to Problem~\ref{prob},
the authors of this paper have developed a recursive method in \cite{Ma} to determine the numbers of connected sets in $(3\times n)$- and $(4\times n)$-lattice graphs, accompanied by the explicit enumeration formula for $N(K_{m}\times P_{n})$ (see Theorem~\ref{thm11}) by employing the transfer matrix method, where $K_m$ is the complete graph of order $m$.

In this paper, we continue the study initiated in~\cite{Ma,Vince2020} by further investigating the enumeration problem on connected sets in (cylindrical) lattice graphs. We first apply an approach in Section~\ref{secmain1} similar to that in~\cite{Ma} to derive the enumeration formula for $N(C_{m}\times P_{n})$ with $m\le 7$ (in Theorem~\ref{main4}), and provide detailed calculation results for the cases $m=4,5$ as examples in Section~\ref{seccal} (see Corollaries~\ref{main1} and~\ref{main2}). Since these enumeration formulas are not in closed form, we further focus on establishing explicit lower bounds of $N(G\times P_n)$ for any graph $G$ (see Theorem~\ref{main3} below) by the transfer matrix method.

For any connected graph $G$, we call a set in $G\times P_n$ consisting of all the vertices with the same second coordinate a \textit{column}. Then any column induces a copy of $G$ in $G\times P_n$, and two columns $X,Y$ are said to be \textit{consecutive} if there exist $(t,u)\in X$ and $(t,v)\in Y$ such that $uv\in E(P_n)$. 
In the following, we shall focus on a special class of connected sets $\mathcal{C}_L(G\times P_n)$, which is defined as
$\{C \in \mathcal C(G\times P_n):
C \cap (X\cup Y)\in \mathcal{C}(G\times P_n)
 \text{ for any two consecutive columns } X,Y~\text{of}~G\times P_n~\text{such that}~
X\cap C, Y\cap C\neq\emptyset\}$. Such sets are more densely connected, as illustrated in Figure~\ref{f12}.

\begin{figure}[h]
\centering
\scalebox{0.65}[0.65]{\includegraphics{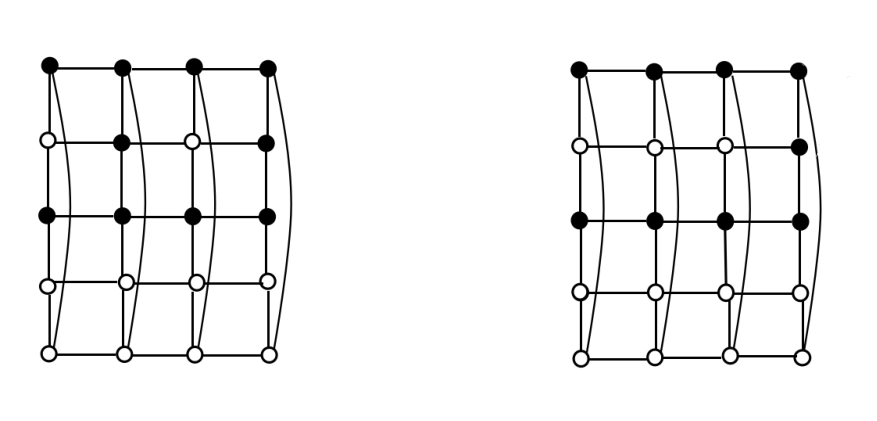}}\\
\hspace{-0.5cm} (a) $C\in \mathcal{C}_L(C_5\times P_4)$ \hspace{2.2cm} (b) $C\notin \mathcal{C}_L(C_5\times P_4)$
\caption{Examples for sets $C$ in $\mathcal{C}(C_5\times P_4)$, where $C$ is the set of solid vertices}
\label{f12}
\end{figure}

Now let $N_L(G\times P_n)=|\mathcal{C}_L(G\times P_n)|$.
Then $N_L(G\times P_n)\le N(G\times P_n)$ as $\mathcal{C}_L(G\times P_n)\subseteq \mathcal{C}(G\times P_n)$.
Moreover, $N_L(G\times P_n)$ is a tight lower bound of $N(G\times P_n)$ and its explicit enumeration formula is as follows.

 \begin{theorem}\label{main3}
Let $G$ be a connected graph of order $m$. Then
 \begin{equation}\label{equ1.2}
 	N(G\times P_n)\ge N_{L}(G\times P_n)=n\cdot N(G)+\sum\limits_{k=2}^{n}(n-k+1)\cdot(\mathbf{1}^{T}\cdot A^{k-1}\mathbf{1} ),
 \end{equation}
where $A=(a_{ij})_{i,j\in \{1,2,...,2^{m}-1\}}$, $ \phi$
 is a fixed bijection from $\{1,2,...,2^{m}-1\}$ to $2^{V(G)}\setminus \{\emptyset\}$, $uv\in E(P_n)$, and
\[
a_{ij} =
\begin{cases}
1, & \text{if}~\{(p,u): p\in \phi(i)\}\cup \{(q,v): q\in \phi(j)\}\in\mathcal{C}(G\times P_n), \\
 0, & \text{otherwise}.
\end{cases}
\]
Moreover, the equality in (\ref{equ1.2}) holds only when $G$ is a complete graph or $n\in\{1,2\}$.
\end{theorem}

The proof of Theorem~\ref{main3} will be given in Section~\ref{secmain2}.

By Theorem~\ref{main3}, we further provide in Section~\ref{secmain2} an asymptotic analysis on the behavior of $N(P_3\times P_n)$, $N(C_4\times P_n)$, $N(C_5\times P_n)$ and $N(K_{1,3}\times P_n)$ 
as $n$ tends to infinity, where $K_{1,3}$ is the star graph of order four. Note that a similar analysis can also be done for other Cartesian product graphs $G\times P_n$, although the calculations become more time-consuming as $|V(G)|$ increases.


\section{The enumeration formula for $N(C_{m}\times P_{n})$ with $m\le7$
\label{secmain1}}
In this section, we shall determine the enumeration formula for $N(C_{m}\times P_{n})$ with $m\le7$.


For $G=P_m\times P_n$ or $C_m\times P_n$, it is clear to see that $\mathcal C(G)\subseteq \mathcal C(K_{m}\times P_{n} )$ as $G$ is a subgraph of $K_{m}\times P_{n}$. Since to determine the enumeration formula for $N(K_{m}\times P_{n} )$ is rather simple (see Theorem~\ref{thm11}), a natural approach here is to identify $\mathcal C(G)$ by removing some elements from $\mathcal C(K_{m}\times P_{n} )$, in which any such element must contain a pair of vertices that are adjacent in $K_{m}\times P_{n}$ but not adjacent in $G$ (the reverse may not be true).
This idea has been applied to derive the enumeration formula for $N(P_m\times P_n)$ with $m\le 4$ in \cite{Ma}. In the following, we shall extend this approach to
obtain similar results for $N(C_m\times P_n)$ with $m\le 7$ based on the background graph $K_{m}\times P_{n}$.\par

 We first fix a labeling of $V(K_{m}\times P_{n})$ for all $m,n$, which naturally induces a labeling of $V(C_{m}\times P_{n})$ and will be used throughout Sections~\ref{secmain1} and~\ref{seccal}.
  Assume that $V(K_{m}\times P_{n})=\{ v_{i,j}: i\in \{1,2,...,m\}, j \in \{1,2,...,n\}\}$, where for any $i_1,i_2\in \{1,2,...,m\}$ and $j_1,j_2\in \{1,2,...,n\}$, $v_{i_1,j_1}v_{i_2,j_1}\in E(K_{m}\times P_{n})$, and $v_{i_1,j_1}v_{i_1,j_2}\in E(K_{m}\times P_{n})$ if and only if $|j_1-j_2|=1$. See an example in Figure~\ref{fig2}. Then for each $j=1,2,\dots, n$, $I_j=\{v_{i,j}:i \in \{1,2,...,m\}\}$ is a column, and we call it the \textit{$j$-th column}. For any $S_{j}\subseteq I_{j}$ with $j\ge 2$,
let $S_{j-1}$ be the set of vertices in $I_{j-1}$ that are adjacent to the vertices in $S_{j}$. Then $|S_{j}|=|S_{j-1}|$ clearly holds.
\par
\begin{figure}[h]
\centering
\scalebox{0.45}[0.45]{\includegraphics{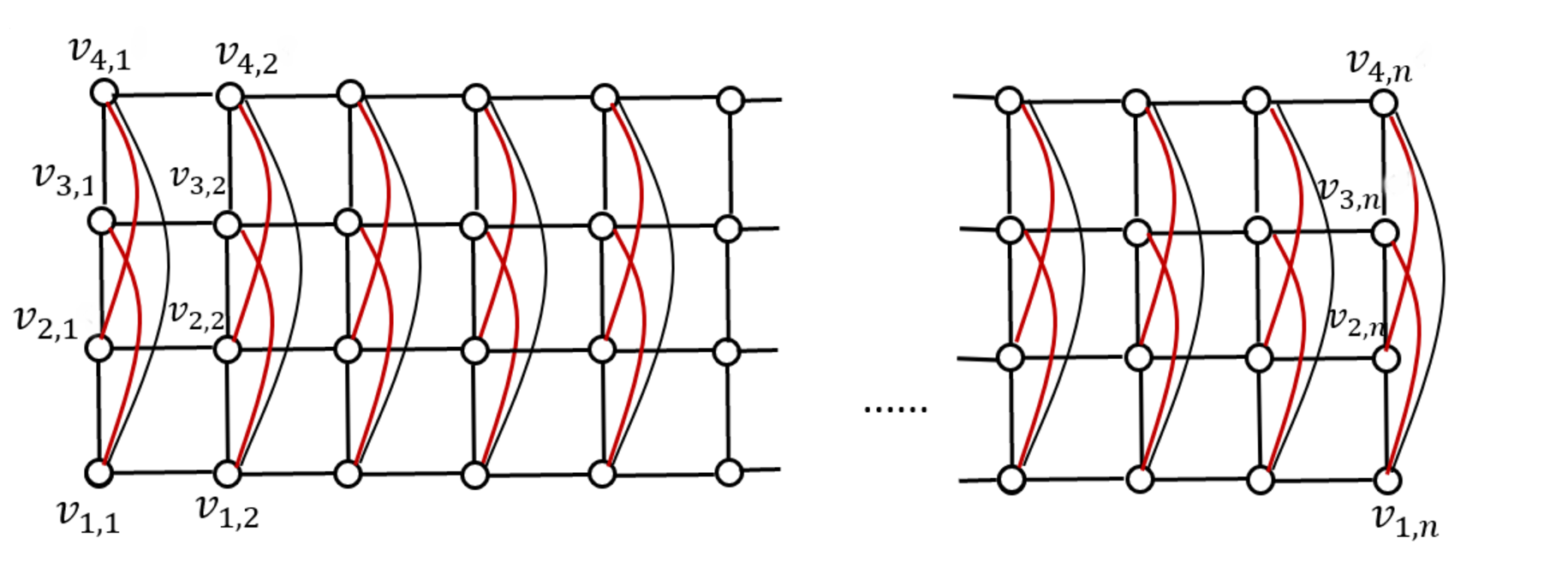}}
\caption{The graph $K_{4}\times P_{n}$}
\label{fig2}
\end{figure}

Let \( \mathcal{C}_{m,n} \) denote the subset of \(\mathcal{C}( K_{m} \times P_{n}) \) that contain at least one vertex from each column. Then it is easy to see that $N(K_{m}\times P_{n})=\sum\limits_{k=1}^{n}(n-k+1) \cdot |\mathcal{C}_{m,k}|.$
For any $i=1,2,\dots,m$, let $f_{m,n}^{i}$ denote the number of sets in \( \mathcal{C}_{m,n} \) that contain a fixed
set of $i$ vertices in the $n$-th column. Then as illustrated in Theorem 1.2 of~\cite{Ma}, $|\mathcal{C}_{m,n}|=\sum \limits_{i=1}^{m}\binom{m} {i}f_{m,n}^{i}$, and $f_{m,n}^{i}$ can be derived by the recurrence formula that
\begin{equation}\label{eqfmn}
f_{m,n}^{i}= |\mathcal{C}_{m,n-1}|-\sum \limits_{j=1}^{m-i}\binom{m-i} {j}f_{m,n-1}^{j},
\end{equation}
with
 $f_{m,1}^{i}=1$ for all $i=1,2,\dots,m$.
 Hence, we present the counting formula for $N(K_{m}\times P_{n})$.
 \begin{theorem}[\cite{Ma}]\label{thm11}
For any positive integers $m$ and $n$,
$$N(K_{m}\times P_{n})=\sum\limits_{k=1}^{n}(n-k+1)\cdot
|\mathcal{C}_{m,k}|,$$
where $|\mathcal{C}_{m,k}|= \begin{bmatrix}\binom{m} {1} & \binom{m} {2} & \cdots & \binom{m} {m}\end{bmatrix} \cdot T^{k-1}\cdot \mathbf{1}$, $T=(t_{ij})_{i,j\in \{1,2,...,m\}}$, and
\[
t_{ij} =
\begin{cases}
\binom{m} {j}-\binom{m-i} {j}, & \text{when } j\leq m-i, \\
 \binom{m} {j}, & \text{when } j\geq m-i+1.
\end{cases}
\]
\end{theorem}

By Theorem~\ref{thm11}, we shall now focus on the set \( \mathcal{C}(K_{m} \times P_{n}) \setminus \mathcal{C}(C_{m} \times P_{n}) \).
Let \( \mathcal{C}_{m,n}' \) be the subset of \( \mathcal{C}(K_{m} \times P_{n}) \setminus \mathcal{C}(C_{m} \times P_{n}) \), where each element of \( \mathcal{C}_{m,n}' \) contains at least one vertex from every column of \( K_{m} \times P_{n} \). Clearly, $\mathcal C_{m,n}'\subseteq \mathcal C_{m,n}$.
For any non-empty set $X\subseteq V(K_{m}\times P_{n})$, let
\begin{eqnarray}
	N(K_{m}\times P_{n};X)&=&|\{C\in \mathcal{C}_{m,n}: C\cap I_i=X\cap I_i~\text{whenever}~X\cap I_i\neq\emptyset\}|,\nonumber\\
	N'(K_{m}\times P_{n};X)&=&|\{C\in \mathcal{C}_{m,n}': C\cap I_i=X\cap I_i~\text{whenever}~X\cap I_i\neq\emptyset\}|.\nonumber
\end{eqnarray}
Then $N(K_{m}\times P_{n};X)=f_{m,n}^{|X|}$ when $X\subseteq I_{n}$, and
the following lemma is easy to observe.
\begin{lemma}\label{lem1}
For any positive integers $m$ and $n$,
\begin{equation}\label{equ2.1}
|\mathcal{C}_{m,n}'|=\sum_{\emptyset\neq S_n\subseteq I_n}N'(K_{m}\times P_{n};S_n),
\end{equation}
where
\begin{equation}\label{equ2.2}
 N'(K_{m}\times P_{n};S_n)=\sum_{\emptyset\neq X\subseteq I_{n-1} \atop X\cap S_{n-1}\neq \emptyset
 }N'(K_{m}\times P_{n};S_n\cup X)~\text{when}~n\ge 2.
 \end{equation}
 \end{lemma}

By (\ref{equ2.1}) and (\ref{equ2.2}), we are able to calculate the value of $|\mathcal{C}_{m,n}'|$ when $n=1,2$. For the case $n\ge 3$, we shall deal with
the following recursive lemma, which can be derived in a manner parallel to Lemma 3.2 in \cite{Ma} and the proof is thus omitted. Here we denote by $k(G)$  the number of components in $G$ for any graph $G$.
   \begin{lemma}\label{lem2}
  	For any $n\ge 3$ and $C \in \mathcal C_{m,n}'$ with $C\cap I_{n}= S_{n}$ and $C\cap I_{n-1}= T_{n-1}$,
  	\begin{enumerate}
  		\item[(1)] if some component of $C_{m}\times P_{n}[S_{n}]$ is also a component of $C_{m}\times P_{n}[S_{n}\cup T_{n-1}]$,
  		then
  		$N'(K_{m}\times P_{n}; S_{n}\cup T_{n-1})=N(K_{m}\times P_{n-1}; T_{n-1} )=f_{m,n-1}^{|T_{n-1}|}$;
  		\item[(2)] if each component of $C_{m}\times P_{n}[S_{n}]$ is not a component of $C_{m}\times P_{n}[S_{n}\cup T_{n-1}]$ and $k(C_{m}\times P_{n}[S_{n}\cup T_{n-1}])=k(C_{m}\times P_{n-1}[ T_{n-1}])$, then $N'(K_{m}\times P_{n}; S_{n}\cup T_{n-1} )=N'(K_{m}\times P_{n-1}; T_{n-1} )$;
  		\item[(3)] for the remaining case that each component of $C_{m}\times P_{n}[S_{n}]$ is not a component of $C_{m}\times P_{n}[S_{n}\cup T_{n-1}]$ and
  		$k(C_{m}\times P_{n}[S_{n}\cup T_{n-1}])< k(C_{m}\times P_{n-1}[ T_{n-1}])$, we have
  		\begin{equation}\label{equ2.3}
  			N'(K_{m}\times P_{n}; S_{n}\cup T_{n-1} )= \sum_{\emptyset\neq X\subseteq I_{n-2} \atop X\cap T_{n-2}\neq \emptyset
  			}N'(K_{m}\times P_{n}; S_{n}\cup T_{n-1} \cup X).
  		\end{equation}
  		Moreover,
  			\begin{enumerate}
  			\item[(i)] if some component of $C_{m}\times P_{n}[S_{n}\cup T_{n-1}]$ is also a component of $C_{m}\times P_{n}[S_{n}\cup T_{n-1}\cup X]$, then $N'(K_{m}\times P_{n}; S_{n}\cup T_{n-1}\cup X)=N(K_{m}\times P_{n-2};X)=f_{m,n-2}^{|X|}$;
  			\item[(ii)] if each component of $C_{m}\times P_{n}[S_{n}\cup T_{n-1}]$ is not a component of $C_{m}\times P_{n}[S_{n}\cup T_{n-1}\cup X]$ and $k(C_{m}\times P_{n}[S_{n}\cup T_{n-1}\cup X])=k(C_{m}\times P_{n-2}[X])$, then $N'(K_{m}\times P_{n}; S_{n}\cup T_{n-1}\cup X )=N'(K_{m}\times P_{n-2}; X)$;
  			\item[(iii)]
otherwise, each component of \( C_m \times P_n[S_n \cup T_{n-1}] \) is not a component of \( C_m \times P_n[S_n \cup T_{n-1} \cup X] \) and \( k(C_m \times P_n[S_n \cup T_{n-1} \cup X]) < k(C_m \times P_{n-2}[X]) \).
By the symmetry of $C_m$, assume that \( R_1, R_2, \dots, R_r \) are the components of \( C_m \times P_{n-2}[X] \), where \( k_s < k_t \) when \( s < t \) for all vertices \( v_{k_s, n-2} \) in \( R_s \) and \( v_{k_t, n-2} \) in \( R_t \). For each $i=1,2,\dots,r$, let $\psi(i)=\{j: R_j~\text{lies in the same component as}~R_i~\text{in}~C_m \times P_n[S_n \cup T_{n-1} \cup X]\}$, and let \( l_i \) (or \( \overline{l}_i \)) be the smallest (or largest) integer \( q \) such that \( v_{q, n-2} \) belongs to the component of \( C_m \times P_n[S_n \cup T_{n-1} \cup X] \) containing \( R_i \). If $\psi(i)$ is a set of consecutive numbers for each $i=1,2,\dots,r$, then $N'(K_{m}\times P_{n}; S_{n}\cup T_{n-1}\cup X)=N'(K_{m}\times P_{n-1}; X\cup \{ v_{q,n-1}: l_{i}\le q \le \overline{l}_{i}, i=1,\dots,r\})$.


  		\end{enumerate}

  	\end{enumerate}
  \end{lemma}
\noindent\textbf{Remark.} Note that for Lemma~\ref{lem2} (3), the only exceptional case is characterized by the following:
each component of \( C_{m} \times P_{n}[S_{n} \cup T_{n-1}] \) is not a component of
\( C_{m} \times P_{n}[S_{n} \cup T_{n-1} \cup X] \), $k(C_{m} \times P_{n}[S_{n} \cup T_{n-1} \cup X]) < k(C_{m} \times P_{n-2}[X]),
$
and there exist four components \( R_p, R_i, R_q, R_j \) in \( C_m \times P_{n-2}[X] \) with \( p < i < q < j\) such that $R_p, R_q$ are in the same component of \( C_{m} \times P_{n}[S_{n} \cup T_{n-1} \cup X] \) while $R_i, R_j$ are in other components.
An example is shown in Figure~\ref{fig3}.
For such cases, it is difficult to find a recurrence formula to calculate \( N'(K_m \times P_n; S_n \cup T_{n-1} \cup X) \).
However, for the case \( m \leq 7 \), we have $k(C_m \times P_{n-2}[X])\le 3$ for all $X\subseteq I_{n-2}$, which implies that the situation described above does not arise.
\begin{figure}[h]
\centering
\scalebox{0.7}[0.7]{\includegraphics{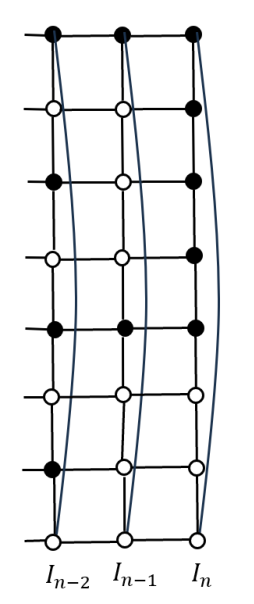}}
\caption{An example of the case that Lemma~\ref{lem2} can not deal with, where $C$ is the set of solid vertices}
\label{fig3}
\end{figure}

Now it is clear to see that Lemmas~\ref{lem1} and~\ref{lem2} together provide a lower bound of $|{\cal C}_{m,n}'|$ for all $m,n$, which can be calculated by multi-step recurrence formulas involving (\ref{eqfmn}).
Then by Theorem~\ref{thm11}, we obtain a sharp upper bound for $ N(C_{m}\times P_{n})$.
\begin{theorem}\label{main4}
For any positive integers $m$ and $n$,	
$$ N(C_{m}\times P_{n})\le \sum\limits_{k=1}^{n}(|{\cal C}_{m,k}|-|{\cal C}_{m,k}'|)\cdot(n-k+1),$$
where the equality holds when $m\le 7$.
\end{theorem}

\section{Calculation results on $N(C_4\times P_n)$ and $N(C_5\times P_n)$
\label{seccal}}
In this section, we shall present detailed calculation results for Theorem~\ref{main4} when $m=4,5$.

In what follows, we abbreviate $N(K_{m}\times P_{n}; \{x_{1},x_{2},...,x_{j}\})$ and $N'(K_{m}\times P_{n}; \mathcal \{x_{1},x_{2},...,x_{j}\})$ as
$N(K_{m}\times P_{n}; x_{1},x_{2},...,x_{j})$ and $N'(K_{m}\times P_{n}; x_{1},x_{2},...,x_{j})$, respectively.

We first focus on $N(C_4\times P_n)$.
For any positive integer $n$, assume that $v_{i+4,n}=v_{i,n}$ for all $i=1,2,3,4$, and let
\begin{itemize}
  \item $a_{n}=N'(K_{4}\times P_{n}; v_{i,n})$ for $i=1,2,3,4$;
  \item $b_{n}^{1}=N'(K_{4}\times P_{n}; v_{i,n}, v_{i+1,n})$ and
      $b_{n}^{2}=N'(K_{4}\times P_{n}; v_{i,n}, v_{i+2,n})$ for $i=1,2,3,4$;
  \item $c_{n}=N'(K_{4}\times P_{n}; v_{i,n}, v_{i+1,n}, v_{i+2,n})$ for $i=1,2,3,4$;
  \item $d_{n}=N'(K_{4}\times P_{n}; I_n)$;
\item $x_n=N'(K_{4}\times P_{n};I_n\cup\{v_{1,n-1},v_{3,n-1}\})$ when $n\ge 2$.
\end{itemize}

Then the following corollary is a special case of Theorem~\ref{main4}.

\renewcommand{\labelenumi}{\rm(\roman{enumi})}
\begin{corollary}\label{main1}
	For any positive integer $n$,
	$$ N(C_{4}\times P_{n})=\sum\limits_{k=1}^{n}(|\mathcal{C}_{4,k}|-|\mathcal{C}_{4,k}'|)\cdot(n-k+1),$$
where
\begin{enumerate}
\item
$$	|\mathcal{C}_{4,k}|=\begin{bmatrix}\binom{4} {1} & \binom{4} {2} & \binom{4} {3} & \binom{4} {4}\end{bmatrix} \cdot
\left[ \begin{matrix}
				1 &3 &3 &1 \\
				2 & 5 &4& 1 \\
				3 & 6 &4& 1 \\
				4 & 6 &4& 1
			\end{matrix} \right]^{k-1}\cdot \mathbf{1};
		$$
\item
	$$|\mathcal{C}_{4,k}'|=4 a_{k}+4 b_{k}^{1}+2b_{k}^{2}+4 c_{k}+d_{k},$$
with	$a_{1}=b_{1}^{1}=c_{1}=d_{1}=x_{2}=0$, $b_{1}^{2}=1$, and when $k\ge 2$,
	\begin{align}\left\{\begin{aligned}
		a_{k}&=a_{k-1}+2b_{k-1}^{1}+b_{k-1}^{2}+3c_{k-1}+d_{k-1},\\
			b_{k}^{1}&=2a_{k-1}+3b_{k-1}^{1}+2b_{k-1}^{2}+4c_{k-1}+d_{k-1},\\
			b_{k}^{2}&=2f_{4,k-1}^{1}+4f_{4,k-1}^{2}+b_{k-1}^{2}+2c_{k-1}+2f_{4,k-1}^{3}+d_{k-1},\\
		c_{k}&=3a_{k-1}+4b_{k-1}^{1}+b_{k-1}^{2}+x_{k}+4c_{k-1}+d_{k-1},\\
			d_{k}&= 4a_{k-1}+4b_{k-1}^{1}+2x_{k}+4c_{k-1}+d_{k-1},\\
		x_k&=2a_{k-2}+4b_{k-2}^{1}+x_{k-1}+4c_{k-2}+d_{k-2} ~\text{where}~k\ge 3.\nonumber
		\end{aligned}\right.
	\end{align}
	\end{enumerate}
\end{corollary}


	


\begin{proof}
 By Theorems~\ref{thm11} and~\ref{main4},
we only need to verify (ii).
Then due to (\ref{equ2.1}), it suffices to verify
the formulas for $a_{k}, b_{k}^{1}, b_{k}^{2}, c_{k}, d_{k}$ along with their respective initial values. The case when $k=1$ is obvious. Assume $k\ge 2$ in the following.
	\par
		
By (\ref{equ2.2}) and Lemma~\ref{lem2} (2), we have
	\begin{eqnarray}
		a_{k}&=& N'(K_{4}\times P_{k}; v_{1,k})\nonumber\\
		&=&N'(K_{4}\times P_{k}; v_{1,k},v_{1,k-1})+ N'(K_{4}\times P_{k}; v_{1,k},v_{1,k-1}, v_{2,k-1})\nonumber\\
		&&+ N'(K_{4}\times P_{k}; v_{1,k},v_{1,k-1}, v_{4,k-1})+ N'(K_{4}\times P_{k}; v_{1,k},v_{1,k-1}, v_{3,k-1})\nonumber\\
		&&+ N'(K_{4}\times P_{k}; v_{1,k},v_{1,k-1}, v_{2,k-1},v_{3,k-1})+ N'(K_{4}\times P_{k}; v_{1,k},v_{1,k-1}, v_{2,k-1},v_{4,k-1})\nonumber\\
		&&N'(K_{4}\times P_{k}; v_{1,k},v_{1,k-1}, v_{3,k-1},v_{4,k-1})+ N'(K_{4}\times P_{k}; \{v_{1,k}\}\cup I_{k-1})\nonumber\\
		&=&N'(K_{4}\times P_{k-1}; v_{1,k-1})+ N'(K_{4}\times P_{k}; v_{1,k-1}, v_{2,k-1})\nonumber\\
		&&+ N'(K_{4}\times P_{k-1}; v_{1,k-1}, v_{4,k-1})+ N'(K_{4}\times P_{k-1};v_{1,k-1}, v_{3,k-1})\nonumber\\
		&&+ N'(K_{4}\times P_{k-1}; v_{1,k-1}, v_{2,k-1},v_{3,k-1})+ N'(K_{4}\times P_{k-1}; v_{1,k-1}, v_{2,k-1},v_{4,k-1})\nonumber\\
		&&N'(K_{4}\times P_{k-1}; v_{1,k-1}, v_{3,k-1},v_{4,k-1})+ N'(K_{4}\times P_{k-1}; I_{k-1})\nonumber\\
		&=&a_{k-1}+2b_{k-1}^{1}+b_{k-1}^{2}+3c_{k-1}+d_{k-1}.\nonumber
	\end{eqnarray}
	
	Similarly, we deduce that
	\begin{eqnarray}
		b_{k}^{1}&= &N'(K_{4}\times P_{k}; v_{1,k}, v_{2,k})\nonumber\\
		&=&N'(K_{4}\times P_{k-1}; v_{1,k-1})
		+N'(K_{4}\times P_{k-1}; v_{2,k-1})
		+ N'(K_{4}\times P_{k-1}; v_{1,k-1}, v_{2,k-1})\nonumber\\
		&&+ N'(K_{4}\times P_{k-1}; v_{1,k-1}, v_{4,k-1})+ N'(K_{4}\times P_{k-1}; v_{2,k-1}, v_{3,k-1})\nonumber\\
		&&+N'(K_{4}\times P_{k-1}; v_{1,k-1}, v_{3,k-1})+N'(K_{4}\times P_{k-1}; v_{2,k-1}, v_{4,k-1})\nonumber\\
		&&+N'(K_{4}\times P_{k-1}; v_{1,k-1},v_{2,k-1},v_{3,k-1})+N'(K_{4}\times P_{k-1}; v_{2,k-1},v_{3,k-1},v_{4,k-1})\nonumber\\
		&&+N'(K_{4}\times P_{k-1}; v_{1,k-1}, v_{3,k-1},v_{4,k-1})+N'(K_{4}\times P_{k-1}; v_{1,k-1},v_{4,k-1},v_{2,k-1})\nonumber\\
		&&+ N'(K_{4}\times P_{k-1}; I_{k-1})\nonumber\\
		&=&2a_{k-1}+3b_{k-1}^{1}+2b_{k-1}^{2}+4c_{k-1}+d_{k-1}.\nonumber
	\end{eqnarray}
		By (\ref{equ2.2}) and Lemma~\ref{lem2} (1), (2), we have
	\begin{eqnarray}
		b_{k}^{2}&=& N'(K_{4}\times P_{k}; v_{1,k}, v_{3,k} )\nonumber\\
		&=&N'(K_{4}\times P_{k}; v_{1,k}, v_{3,k} , v_{1,k-1})+ N'(K_{4}\times P_{k}; v_{1,k}, v_{3,k} , v_{3,k-1})\nonumber\\
		&&+N'(K_{4}\times P_{k}; v_{1,k}, v_{3,k} , v_{1,k-1}, v_{2,k-1})
		+ N'(K_{4}\times P_{k}; v_{1,k}, v_{3,k}, v_{1,k-1}, v_{4,k-1})\nonumber\\
		&&+N'(K_{4}\times P_{k}; v_{1,k}, v_{3,k} , v_{2,k-1}, v_{3,k-1})+ N'(K_{4}\times P_{k}; v_{1,k}, v_{3,k}, v_{3,k-1}, v_{4,k-1})\nonumber\\
		&&+N' (K_{4}\times P_{k}; v_{1,k}, v_{3,k}, v_{1,k-1}, v_{3,k-1})+N'(K_{4}\times P_{k}; v_{1,k}, v_{3,k},v_{1,k-1},v_{2,k-1},v_{3,k-1})\nonumber\\
&&+N'(K_{4}\times P_{k}; v_{1,k}, v_{3,k}, v_{2,k-1},v_{3,k-1},v_{4,k-1})+N'(K_{4}\times P_{k}; v_{1,k}, v_{3,k}, v_{3,k-1},v_{4,k-1},v_{1,k-1})\nonumber\\
		&&+N'(K_{4}\times P_{k}; v_{1,k}, v_{3,k}, v_{4,k-1},v_{1,k-1},v_{2,k-1})+N'(K_{4}\times P_{k}; \{v_{1,k}, v_{3,k}\} \cup I_{k-1})\nonumber\\
		&=&N (K_{4}\times P_{k-1}; v_{1,k-1})+ N(K_{4}\times P_{k-1}; v_{3,k-1})\nonumber\\
		&&+N(K_{4}\times P_{k-1}; v_{1,k-1}, v_{2,k-1})+ N(K_{4}\times P_{k-1}; v_{1,k-1}, v_{4,k-1})\nonumber\\
	 &&+N(K_{4}\times P_{k-1}; v_{2,k-1}, v_{3,k-1})+ N(K_{4}\times P_{k-1}; v_{3,k-1}, v_{4,k-1})\nonumber\\
	 &&+N' (K_{4}\times P_{k-1}; v_{1,k-1}, v_{3,k-1})+N'(K_{4}\times P_{k-1}; v_{1,k-1},v_{2,k-1},v_{3,k-1})\nonumber\\
&&+N(K_{4}\times P_{k-1}; v_{2,k-1},v_{3,k-1},v_{4,k-1})+N'(K_{4}\times P_{k-1}; v_{3,k-1},v_{4,k-1},v_{1,k-1})\nonumber\\
		&&+N(K_{4}\times P_{k-1}; v_{4,k-1},v_{1,k-1},v_{2,k-1})+N'(K_{4}\times P_{k-1}; I_{k-1})\nonumber\\
		&=&2f_{4,k-1}^{1}+4f_{4,k-1}^{2}+b_{k-1}^{2}+2f_{4,k-1}^{3}+2c_{k-1}+d_{k-1},\nonumber
	\end{eqnarray}
where $f_{4,k-1}^{i}$ with $i=1,2,3$ can be calculated recursively by (\ref{eqfmn}).
	
 Next, by (\ref{equ2.2})
 and Lemma~\ref{lem2} (2), (3),
  we have
 	\begin{eqnarray}
 c_{k}&=&N'(K_{4}\times P_{k}; v_{1,k},v_{2,k}, v_{3,k})\nonumber\\
 &=&N'(K_{4}\times P_{k-1}; v_{1,k-1})+N'(K_{4}\times P_{k-1}; v_{2,k-1})+N'(K_{4}\times P_{k-1}; v_{3,k-1})\nonumber\\
 &&+N'(K_{4}\times P_{k-1};v_{1,k-1},v_{2,k-1})+x_n+N'(K_{4}\times P_{k-1}; v_{1,k-1},v_{4,k-1})\nonumber\\
 &&+N'(K_{4}\times P_{k-1}; v_{2,k-1},v_{3,k-1})+N'(K_{4}\times P_{k-1}; v_{3,k-1},v_{4,k-1})\nonumber\\
 &&+N'(K_{4}\times P_{k-1}; v_{2,k-1},v_{4,k-1})+N'(K_{4}\times P_{k-1}; v_{1,k-1},v_{2,k-1},v_{3,k-1})\nonumber\\
 &&+N'(K_{4}\times P_{k-1}; v_{1,k-1},v_{2,k-1},v_{4,k-1})+N'(K_{4}\times P_{k-1}; v_{2,k-1},v_{3,k-1},v_{4,k-1})\nonumber\\
 &&+
 N'(K_{4}\times P_{k-1}; v_{1,k-1},v_{3,k-1},v_{4,k-1})+N'(K_{4}\times P_{k-1}; I_{k-1})\nonumber\\
 &=&3a_{k-1}+4b_{k-1}^{1}+b_{k-1}^{2}+x_{k}+4c_{k-1}+d_{k-1},\nonumber
 \end{eqnarray}
 and
 \begin{eqnarray}
 d_{k}&=&N'(K_{4}\times P_{k}; I_{k})=4a_{k-1}+4b_{k-1}^{1}+2x_{k}+4c_{k-1}+d_{k-1}.\nonumber
 \end{eqnarray}

Then it remains to verify the recurrence formula for $x_k$. Obviously, $x_2=0$. For $k\ge 3$, by
 Lemma~\ref{lem2} (3), we have
	\begin{eqnarray}
		x_{k}&=&N'(K_{4}\times P_{k};I_k\cup\{v_{1,k-1},v_{3,k-1}\})\nonumber\\
		&=&N'(K_{4}\times P_{k};I_k\cup\{v_{1,k-1},v_{3,k-1},v_{1,k-2}\})+N'(K_{4}\times P_{k};I_k\cup\{v_{1,k-1},v_{3,k-1},v_{3,k-2}\})\nonumber\\
		&&+N'(K_{4}\times P_{k};I_k\cup\{v_{1,k-1},v_{3,k-1}, v_{1,k-2},v_{2,k-2}\})\nonumber\\
		&&+N'(K_{4}\times P_{k};I_k\cup\{v_{1,k-1},v_{3,k-1},v_{1,k-2}, v_{4,k-2}\})\nonumber\\
		&&+N'(K_{4}\times P_{k};I_k\cup\{v_{1,k-1},v_{3,k-1},v_{2,k-2},v_{3,k-2}\})\nonumber\\
		&&+N'(K_{4}\times P_{k};I_k\cup\{v_{1,k-1},v_{3,k-1},v_{3,k-2},v_{4,k-2}\})\nonumber\\
		&&+N'(K_{4}\times P_{k};I_k\cup\{v_{1,k-1},v_{3,k-1}, v_{1,k-2},v_{3,k-2}\})\nonumber\\
		&&+N'(K_{4}\times P_{k};I_k\cup\{v_{1,k-1},v_{3,k-1}, v_{1,k-2},v_{2,k-2},v_{3,k-2}\})\nonumber\\
		&&+N'(K_{4}\times P_{k};I_k\cup\{v_{1,k-1},v_{3,k-1}, v_{1,k-2},v_{2,k-2},v_{4,k-2}\})\nonumber\\
		&&+N'(K_{4}\times P_{k};I_k\cup\{v_{1,k-1},v_{3,k-1}, v_{1,k-2},v_{3,k-2},v_{4,k-2}\})\nonumber\\
		&&+N'(K_{4}\times P_{k};I_k\cup\{v_{1,k-1},v_{3,k-1}, v_{2,k-2},v_{3,k-2},v_{4,k-2}\})\nonumber\\
		&&+N'(K_{4}\times P_{k};I_k\cup\{v_{1,k-1},v_{3,k-1}\}\cup I_{k-2}),\nonumber
	\end{eqnarray}
where
	\begin{eqnarray}
		&&N'(K_{4}\times P_{k}; I_{k}\cup \{v_{1,k-1},v_{3,k-1},v_{1,k-2},v_{3,k-2}\})\nonumber\\
		&=&N'(K_{4}\times P_{k-1}; I_{k-1}\cup \{v_{1,k-2},v_{3,k-2}\})\nonumber\\
		&=&x_{k-1}.\nonumber
	\end{eqnarray}
Thus,
	\begin{eqnarray}
		x_k=2a_{k-2}+4b_{k-2}^{1}+x_{k-1}+4c_{k-2}+d_{k-2}. \nonumber
	\end{eqnarray}

	The proof is complete.
\end{proof}

Similarly, we establish the recurrence formulas for $N(C_5\times P_n)$. For any positive integer $n$, assume that $v_{i+5,n}=v_{i,n}$ for all $i=1,2,3,4,5$, and let
\begin{itemize}
  \item $a_{n}=N'(K_{5}\times P_{n}; v_{i,n})$ for $i=1,2,3,4,5$;
  \item $b_{n}^{1}=N'(K_{5}\times P_{n}; v_{i,n}, v_{i+1,n})$ and $b_{n}^{2}=N'(K_{5}\times P_{n}; v_{i,n}, v_{i+2,n})$ for $i=1,2,3,4,5$;
  \item $c_{n}^{1}=N'(K_{5}\times P_{n}; v_{i,n}, v_{i+1,n}, v_{i+2,n})$ and $c_{n}^{2}=N'(K_{5}\times P_{n}; v_{i,n}, v_{i+2,n}, v_{i+3,n})$ for $i=1,2,3,4,5$;
   \item $d_{n}=N'(K_{5}\times P_{n}; v_{i,n}, v_{i+1,n}, v_{i+2,n}, v_{i+3,n})$ for $i=1,2,3,4,5$;
  \item $g_{n}=N'(K_{5}\times P_{n}; I_n)$;
  \item \( x_n=N'(K_{5}\times P_{n}; v_{1,n}, v_{2,n},v_{3,n}, v_{1,n-1}, v_{3,n-1})\) when $n\ge 2$;
  \item
 \(
   y_n=N'(K_{5}\times P_{n}; v_{1,n}, v_{2,n},v_{3,n}, v_{1,n-1}, v_{3,n-1}, v_{4,n-1}) \\
     \hspace{-8cm}= N'(K_{5}\times P_{n}; v_{1,n}, v_{2,n},v_{3,n}, v_{1,n-1}, v_{3,n-1}, v_{5,n-1}) \) when $n\ge 2$.

\end{itemize}

Then the next corollary is direct.
\begin{corollary}\label{main2}
	For any positive integer $n$,
	$$ N(C_{5}\times P_{n})=\sum\limits_{k=1}^{n}(|\mathcal{C}_{5,k}|-|\mathcal{C}_{5,k}'|)\cdot(n-k+1),$$
	where
	\begin{enumerate}
	\item
	$$	|\mathcal{C}_{5,k}|=\begin{bmatrix}\binom{5} {1} & \binom{5} {2} & \binom{5} {3} & \binom{5} {4} & \binom{5} {5}\end{bmatrix} \cdot \left[ \begin{matrix}
				1 &4 &6 &4 &1 \\
				2 &7 &9 &5 &1 \\
				3 &9 &10 &5 &1 \\
				4 &10 &10 &5 &1 \\
				5 &10 &10 &5 &1
			\end{matrix} \right]^{k-1}\cdot\mathbf{1};$$
	\item
		$$|\mathcal{C}_{5,k}'|=5( a_{k}+ b_{k}^{1}+b_{k}^{2}+c_{k}^{1}+c_{k}^{2}+d_{k})+g_{k},$$
	with	$a_{1}=b_{1}^{1}=c_{1}^{1}=d_{1}=g_{1}= x_{2}= y_{2}=0, b_{1}^{2}=c_{1}^{2}=1$, and when $k\ge 2$,
	\begin{align}\left\{\begin{aligned}
			a_{k}&=a_{k-1}+2b_{k-1}^{1}+2b_{k-1}^{2}+3c_{k-1}^{1}+3c_{k-1}^{2}+4d_{k-1}+g_{k-1},\\
			b_{k}^{1}&=2a_{k-1}+3b_{k-1}^{1}+4b_{k-1}^{2}+4c_{k-1}^{1}+5c_{k-1}^{2}+5d_{k-1}+g_{k-1},\\
			b_{k}^{2}&=2f_{5,k-1}^{1}+6f_{5,k-1}^{2}+b_{k-1}^{2}+6f_{5,k-1}^{3}+c_{k-1}^{1}+2c_{k-1}^{2}+2f_{5,k-1}^{4}+3d_{k-1}+g_{k-1},\\
			c_{k}^{1}&=3a_{k-1}+4b_{k-1}^{1}+x_{k}+4b_{k-1}^{2}+5c_{k-1}^{1}+3c_{k-1}^{2}+2y_{k}+5d_{k-1}+g_{k-1},\\
			c_{k}^{2}&=3f_{5,k-1}^{1}+7f_{5,k-1}^{2}+2b_{k-1}^{2}+5f_{5,k-1}^{3}+2c_{k-1}^{1}+3c_{k-1}^{2}+f_{5,k-1}^{4}+4d_{k-1}+g_{k-1},\\
			d_{k}&=4a_{k-1}+5b_{k-1}^{1}+2b_{k-1}^{2}+3x_{k}+5c_{k-1}^{1}+4y_{k}+c_{k-1}^{2}+5d_{k-1}+g_{k-1},\\
			g_{k}&=5a_{k-1}+5b_{k-1}^{1}+5x_{k}+5y_{k}+5c_{k-1}^{1}+5d_{k-1}+g_{k-1},\\
			x_{k}&=2a_{k-2}+4b_{k-2}^{1}+x_{k-1}+2b_{k-2}^{2}+5c_{k-2}^{1}+2c_{k-2}^{2}+2y_{k-1}+5d_{k-2}+g_{k-2}~\text{where}~k\ge 3,\\
			y_{k}&=3a_{k-2}+5b_{k-2}^{1}+2x_{k-1}+2b_{k-2}^{2}+5c_{k-2}^{1}+2c_{k-2}^{2}+3y_{k-1}+5d_{k-2}+g_{k-2}~\text{where}~k\ge 3.\nonumber
		\end{aligned}\right.
	\end{align}
		\end{enumerate}
\end{corollary}

\section{Asymptotic analysis
\label{secmain2}}
In this section, we prove Theorem~\ref{main3} and perform an asymptotic analysis on $N(P_3\times P_n)$, $N(C_4\times P_n)$, $N(C_5\times P_n)$ and $N(K_{1,3}\times P_n)$.\par


\noindent\textit{Proof of Theorem~\ref{main3}.}
We begin by proving (\ref{equ1.2}). For each $k=1,2,\dots,n$, assume that  $V(P_k)=\{v_1,v_2,\dots,v_k\}$, where $v_iv_{i+1}\in E(P_k)$ for all $i\in \{1,2,\dots,k-1\}$. Then the $j$-th column $I_j$ of $G\times P_k$ is clearly $\{(x,v_j): x\in V(G)\}$ for all $j\in \{1,2,\dots,k\}$.
Moreover, let
$$\mathcal{S}_L(k)=\{C\in \mathcal{C}_L(G\times P_k): C \cap I_{j}\neq \emptyset~\text{for
each}~j =1,2,...,k\}.$$
Then it suffices to show that $\mathbf{1}^{T}\cdot A^{k-1}\cdot \mathbf{1}=|\mathcal{S}_L(k)|$ for all $k\ge 2$.\par


 In fact, we shall show by induction that for each $(i,j)$-entry in $A^{k-1}$, it is the number of sets in $\mathcal{S}_L(k)$ containing $\{(p,v_1): p\in \phi(i)\}\cup \{(q,v_k): q\in \phi(j)\} $. Then (\ref{equ1.2}) is proven by the definition of $\phi$.

 When $k=2$, the matrix $A$ clearly satisfies our requirements. Suppose the statement is true for $k=t$, where
$t\ge 2$. Let $l_{i,r}$ be any $(i,r)$-entry in $A^{t-1}$. Then $l_{i,r}$ is the number of sets in $\mathcal{S}_L(t)$ containing $\{(p,v_1): p\in \phi(i)\}\cup \{(q,v_t): q\in \phi(r)\} $. Now we consider the case for $k=t+1$. Each $(i,j)$-entry
in the matrix $A^{t}$ is $\sum\limits_{r=1}^{2^{m}-1} l_{i,r} a_{r,j}$, where $a_{r,j}=1$ if and only if
$\{(p,v_t): p\in \phi(r)\}\cup \{(q,v_{t+1}): q\in \phi(j)\}\in \mathcal{C}(G\times P_{t+1}).$
Consequently, $\sum\limits_{r=1}^{2^{m}-1} l_{i,r} a_{r,j}$ is the number of sets in $\mathcal{S}_L(t+1)$ containing $\{(p,v_1): p\in \phi(i)\}\cup \{(q,v_{t+1}): q\in \phi(j)\} $.
Hence the matrix $A^{t}$ satisfies the requirements in the statement and (\ref{equ1.2}) is proven.

Next, we show that (\ref{equ1.2}) is sharp only when $G$ is complete or $n \in \{1, 2\}$.
It is clear that the equality in (\ref{equ1.2}) holds when $G$ is complete or $n \in \{1, 2\}$.
Now, assume that the equality in (\ref{equ1.2}) holds for a connected graph $G$ that is not complete and $n \ge 3$. Then there exist non-adjacent vertices $u$ and $v$ in $G$ connected by a path $u x_1 x_2 \dots x_t v$ in $G$. Let
$$
C = \{(u, v_1), (v, v_1), (u, v_2), (v, v_2), (u, v_3), (x_1, v_3),(x_2, v_3), \dots, (x_t, v_3),(v, v_3)\}.
$$
Since $n\ge 3$, it can be easily verified that $C \in \mathcal{C}(G \times P_n)$ and $C \cap I_{1},C \cap I_{2}\neq \emptyset$. However,
$C\cap (I_{1}\cup I_{2}) \notin \mathcal{C}(G \times P_n) $, implying that $C\in \mathcal{C}(G \times P_n)\setminus \mathcal{C}_L(G \times P_n)$. Hence $N_L(G \times P_n) < N(G \times P_n)$, a contradiction.
As a result, (\ref{equ1.2}) is sharp only when $G$ is complete or $n \in \{1, 2\}$.
\hfill$\Box$

Note that the transfer matrix $A$ given in Theorem~\ref{main3} is symmetric and makes the calculation of an estimate of $N(G\times P_n)$ for any graph $G$ possible. In particular, this kind of estimate works well for (cylindrical) lattice graphs. See several examples below.

\begin{example}\end{example}
\begin{enumerate}
\item For $C_4\times P_n$, we have $A_{1}$ as the matrix $A$ in Theorem~\ref{main3} below:
\[
A_{1} = \begin{pmatrix}
1 & 0 & 0 & 0 & 1 & 0 & 0 & 1 & 0 & 0 & 1 & 0 & 1 & 1 & 1 \\
0 & 1 & 0 & 0 & 1 & 1 & 0 & 0 & 0 & 0 & 1 & 1 & 0 & 1 & 1 \\
0 & 0 & 1 & 0 & 0 & 1 & 1 & 0 & 0 & 0 & 1 & 1 & 1 & 0 & 1 \\
0 & 0 & 0 & 1 & 0 & 0 & 1 & 1 & 0 & 0 & 0 & 1 & 1 & 1 & 1 \\
1 & 1 & 0 & 0 & 1 & 1 & 0 & 1 & 0 & 0 & 1 & 1 & 1 & 1 & 1 \\
0 & 1 & 1 & 0 & 1 & 1 & 1 & 0 & 0 & 0 & 1 & 1 & 1 & 1 & 1 \\
0 & 0 & 1 & 1 & 0 & 1 & 1 & 1 & 0 & 0 & 1 & 1 & 1 & 1 & 1 \\
1 & 0 & 0 & 1 & 1 & 0 & 1 & 1 & 0 & 0 & 1 & 1 & 1 & 1 & 1 \\
0 & 0 & 0 & 0 & 0 & 0 & 0 & 0 & 0 & 0 & 1 & 0 & 1 & 0 & 1 \\
0 & 0 & 0 & 0 & 0 & 0 & 0 & 0 & 0 & 0 & 0 & 1 & 0 & 1 & 1 \\
1 & 1 & 1 & 0 & 1 & 1 & 1 & 1 & 1 & 0 & 1 & 1 & 1 & 1 & 1 \\
0 & 1 & 1 & 1 & 1 & 1 & 1 & 1 & 0 & 1 & 1 & 1 & 1 & 1 & 1 \\
1 & 0 & 1 & 1 & 1 & 1 & 1 & 1 & 1 & 0 & 1 & 1 & 1 & 1 & 1 \\
1 & 1 & 0 & 1 & 1 & 1 & 1 & 1 & 0 & 1 & 1 & 1 & 1 & 1 & 1 \\
1 & 1 & 1 & 1 & 1 & 1 & 1 & 1 & 1 & 1 & 1 & 1 & 1 & 1 & 1
\end{pmatrix}.
\]

By the Gelfand theorem and the fact that \( A _{1}\) is symmetric with all non-negative entries, we have
\[\lim_{n \to \infty} \left( \mathbf{1}^T A_{1}^{n-1} \mathbf{1} \right)^{1/n}= \lambda_{ A_{1}},\]
where $\lambda_{ A_{1}}$ is the largest eigenvalue of \( A_{1} \).
Using Mathematica, we find that \( \lambda_{ A_{1}} \approx 10.5318 \). Hence for sufficiently large $n$,
\[
c(C_{4}\times P_{n}) = N(C_{4}\times P_{n}) ^{1/4n}>  \left( \mathbf{1}^T A_{1}^{n-1} \mathbf{1} \right)^{1/4n}=\lambda_{ A_{1}}^{1/4} \approx 1.8014.
\]

\item Similarly for $C_5\times P_n$, we have \( \lambda_{ A_{2}} \approx 18.2600\) with $A_2$ as the matrix $A$ in Theorem~\ref{main3}. Thus
$$c(C_{5}\times P_{n})=N(C_{5}\times P_{n}) ^{1/5n}> \lambda_{ A_{2}}^{1/5}\approx 1.7877.$$

\item
Moreover, \(c_{1}=(N_L(C_4\times P_n))^{1/4n}\) and \(c_{2}=(N_L(C_5\times P_n))^{1/5n}\) provide a good estimate of  \( c(C_{4}\times P_{n}) \)  and \( c(C_{5}\times P_{n}) \), respectively. See some data in Figure~\ref{fig100},  Table~\ref{table1} and Table~\ref{table2} as examples that illustrate
the asymptotic behavior of \( c(C_{4}\times P_{n}) \),  \( c(C_{5}\times P_{n}) \), \(c_{1}\) and \(c_{2}\).

\item 
Additionally, based on extensive Mathematica computations, we suspect that the limits of $c(C_{4}\times P_{n})$ and $c(C_{5}\times P_{n})$ exist as $n$ tends to infinity, with approximate values 1.8374 and 1.8670, respectively.

\vspace{0.5cm}

\begin{figure}[H]
\centering
\begin{subfigure}[b]{0.45\textwidth}
    \includegraphics[width=\textwidth]{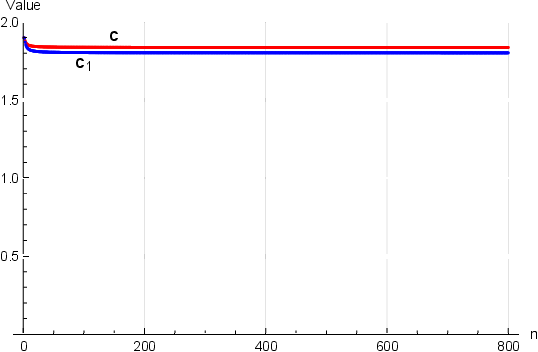}
    \caption{The comparison of $c(C_{4}\times P_{n})$ and $c_{1}$}
\end{subfigure}
\hspace{0.05\textwidth}
\begin{subfigure}[b]{0.45\textwidth}
    \includegraphics[width=\textwidth]{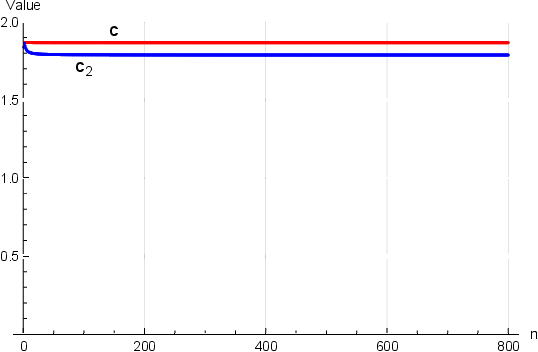}
    \caption{The comparison of $c(C_{5}\times P_{n})$ and $c_{2}$}
\end{subfigure}
\caption{Asymptotic behavior of \( c\), $c_{1}$ and $c_{2}$}
\label{fig100}
\end{figure}

\vspace{15pt}

\begin{table}[H]
\centering
\caption{The values of $c(C_{4}\times P_{n})$ and $c_{1}$ for $1 \le n \le 11$.}
\setlength{\tabcolsep}{2pt}
\renewcommand{\arraystretch}{1.15}
\begin{tabular}{@{}ccccccccccccc@{}}
\toprule
$n$ &1 & 2 & 3 & 4 & 5 & 6 & 7 & 8 & 9 & 10 &11\\
\midrule
$c(C_{4}\times P_{n})$ &
1.8988& 1.8960& 1.8796&1.8690&1.8624&1.8580&1.8548& 1.8525&
1.8506& 1.8492& 1.8480\\
$c_{1}$ &
1.8988& 1.8960& 1.8706& 1.8534&1.8430& 1.8360& 1.8310& 1.8273&
1.8244&1.8221& 1.8202\\
\bottomrule
\end{tabular}
\label{table1}
\end{table}
\vspace{15pt}
\begin{table}[H]
\centering
\caption{The values of $c(C_{5}\times P_{n})$ and $c_{2}$ for $1 \le n \le 11$.}
\setlength{\tabcolsep}{1pt}
\renewcommand{\arraystretch}{1.15}
\begin{tabular}{@{}ccccccccccccccccccccccccccccc@{}}
\toprule
$n$ &1 & 2 & 3 & 4 & 5 & 6 & 7 & 8 & 9 & 10 & 11 \\
\midrule
$c(C_{5}\times P_{n})$ &
1.8384& 1.8628& 1.8687& 1.8687& 1.8683& 1.8681& 1.8679& 1.8678&
1.8677& 1.8676& 1.8676\\
$c_{2}$ &1.8384& 1.8628& 1.8466& 1.8308& 1.8224& 1.8165& 1.8124& 1.8093&
1.8069&1.8049& 1.8034\\
\bottomrule
\end{tabular}
\label{table2}
\end{table}
\end{enumerate}

\vspace{1cm}
\begin{example}
\end{example}
\begin{enumerate}
\item For $P_3\times P_n$, we have $A_{3}$ as the matrix $A$ in Theorem~\ref{main3} below:
\[
A_{3} = \begin{pmatrix}
1 & 0 & 0 & 1 & 0 & 0 & 1 \\
0 & 1 & 0 & 1 & 1 & 0 & 1 \\
0 & 0 & 1 & 0 & 1 & 0 & 1 \\
1 & 1 & 0 & 1 & 1 & 0 & 1 \\
0 & 1 & 1 & 1 & 1 & 0 & 1 \\
0 & 0 & 0 & 0 & 0 & 0 & 1 \\
1 & 1 & 1 & 1 & 1 & 1 & 1
\end{pmatrix}.
\]
By using Mathematica, \( \lambda_{ A_{3}} \approx 4.6524\), and
\[
c(P_{3}\times P_{n}) = N(P_{3}\times P_{n})^{1/3n}> \lambda_{A_{3}}^{1/3} \approx 1.6694.
\]
\item For $K_{1,3}\times P_n$, we have \( \lambda_{ A_{4}} \approx 8.9322\) with $A_4$ as the matrix $A$ in Theorem~\ref{main3}. Thus,
\[
c(K_{1,3}\times P_{n}) = N(K_{1,3}\times P_{n})^{1/4n} > \lambda_{A_{4}}^{1/4} \approx 1.7288.
\]
\item
Moreover, \(c_{3}=(N_L(P_3\times P_n))^{1/3n}\) and $c_4=(N_L(K_{1,3}\times P_n))^{1/4n}$ provide a good estimate of \( c(P_{3}\times P_{n}) \) and  \( c(K_{1,3}\times P_{n}) \), respectively. See some data in Figure~\ref{fig200},  Table~\ref{table3} and Table~\ref{table4} as examples that illustrate
the asymptotic behavior of \( c(P_{3}\times P_{n}) \), \( c(K_{1,3}\times P_{n}) \), \(c_{3}\) and $c_4$.

\item  Additionally, based on extensive Mathematica computations, we suspect that the limits of $c(P_{3}\times P_{n})$  and $c(K_{1,3}\times P_{n})$ exist as $n$ tends to infinity, with approximate values of 1.7196 and 1.7914, respectively.

\begin{figure}[H]
\centering
\begin{subfigure}[b]{0.45\textwidth}
    \includegraphics[width=\textwidth]{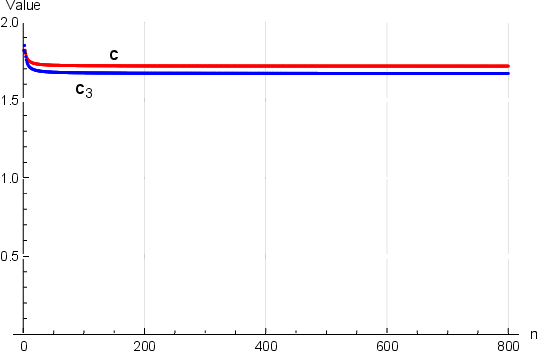}
    \caption{The comparison of $c(P_{3}\times P_{n})$ and $c_{3}$}
\end{subfigure}
\hspace{0.05\textwidth}
\begin{subfigure}[b]{0.45\textwidth}
    \includegraphics[width=\textwidth]{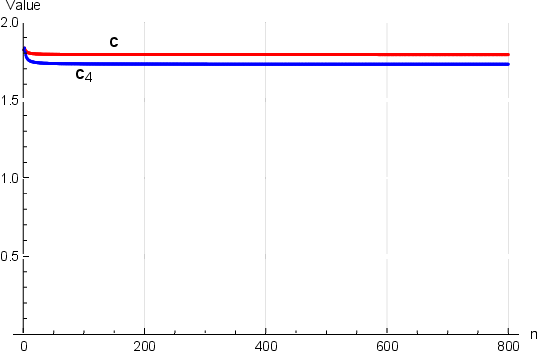}
    \caption{The comparison of $c(K_{1,3}\times P_{n})$ and $c_{4}$}
\end{subfigure}
\caption{Asymptotic behavior of \( c \), $c_{3}$ and \( c_{4}\)}
\label{fig200}
\end{figure}

\vspace{8pt}

\begin{table}[H]
\centering
\caption{The values of $c(P_{3}\times P_{n})$ and $c_{3}$ for $1\le n \le 11$.}
\setlength{\tabcolsep}{1pt}
\renewcommand{\arraystretch}{1.15}
\begin{tabular}{@{}ccccccccccccccccccccccccccccc@{}}
\toprule
$n$ &1& 2 & 3 & 4 & 5 & 6 & 7 & 8 & 9 & 10 & 11 \\
\midrule
$c(P_{3}\times P_{n})$ & 1.8171& 1.8493& 1.8190& 1.7960& 1.7804& 1.7697& 1.7620& 1.7563&
1.7518& 1.7482& 1.7453\\
$c_{3}$ &1.8171&1.8493& 1.8095& 1.7772& 1.7560& 1.7414& 1.7310& 1.7231&
1.7171& 1.7123& 1.7083\\
\bottomrule
\end{tabular}
\label{table3}
\end{table}

\vspace{8pt}
\begin{table}[htbp]
\centering
\caption{The values of $c(K_{1,3}\times P_{n})$ and $c_{4}$ for $1 \le n \le 11$.}
\setlength{\tabcolsep}{1pt}
\renewcommand{\arraystretch}{1.15}
\begin{tabular}{@{}ccccccccccccccccccccccccccccc@{}}
\toprule
$n$ &1 & 2 & 3 & 4 & 5 & 6 & 7 & 8 & 9 & 10 & 11 \\
\midrule
$c(K_{1,3}\times P_{n})$ &1.8212&1.8322&1.8224& 1.8148& 1.8098&1.8065& 1.8042& 1.8025&
1.8011& 1.8001& 1.7992\\
$c_{4}$ & 1.8212& 1.8322& 1.8075& 1.7875& 1.7757& 1.7677& 1.7621&1.7579&
1.7547&1.7521& 1.7499\\
\bottomrule
\end{tabular}
\label{table4}
\end{table}
\end{enumerate}

\section*{Acknowledgement}
\noindent
 This work was supported by National Natural Science Foundation of China (No. 12571379 and 12501498).

\section*{Statements and Declarations}
\begin{itemize}
	\item Competing Interests: The authors have no relevant financial or non-financial interests to disclose.
	\item Data availability: No data were used in this study.
\end{itemize}

\end{document}